\documentclass{article}
\usepackage[english]{babel} %
\usepackage[a4paper, left=4cm, right=4cm, top=3cm, bottom=3cm]{geometry}

\usepackage{amssymb,amsmath,amsthm}
\usepackage{enumitem}
\usepackage[justification=centering]{caption}

\usepackage{tikz,xcolor}
\usetikzlibrary{calc}

\usepackage{listings}
\usepackage{subfactors-figures}

\theoremstyle{plain}
\newtheorem{theorem}{Theorem} 
\newtheorem*{theorem*}{Theorem} 
\newtheorem{lemma}[theorem]{Lemma}
\newtheorem{corollary}[theorem]{Corollary}

\definecolor{xblue}{RGB}{0,0,139}%
\definecolor{xlightgreen}{RGB}{200,230,200}
\definecolor{xgold}{RGB}{255,215,0}

\DeclareMathOperator{\Cay}{Cay}
\DeclareMathOperator{\Aut}{Aut}
\DeclareMathOperator{\Ker}{Ker}
\DeclareMathOperator{\wt}{wt}
\makeatletter
  \newcommand{\smod}[1]{\allowbreak\mkern 6mu({\operator@font mod}\;#1)}
\makeatother
\def\GAP{\texttt{GAP} }

\usepackage[colorlinks=true, allcolors=xblue, bookmarks=false]{hyperref}

\title{Factors in finite groups and well-covered graphs}
\author{Mikhail Kabenyuk}
\date{}

\begin{document}\maketitle

\begin{abstract}
We study a combinatorial property of subsets in finite groups that is analogous 
to the notion of independence in graphs. Given a group $G$ and a non-empty subset $A\subset G$, 
we define  a (right) $s$-factor as a subset $B\subset G$ satisfying the following conditions: 

(i) Every element of $AB$ can be written uniquely as $ab$ with $a\in A$ and $b\in B$. 

(ii) $B$ is maximal (with respect to inclusion) with this property. 

For a finite group $G$, the upper and lower indices of $A$ are 
the sizes of the largest and smallest $s$-factors associated with $A$. 
A subset is called stable if its upper and lower indices coincide.  
A group is called stable if all its subsets are stable. 

We then explore the connection between $s$-factors in groups and maximal independent sets in graphs. 
Specifically, we show that $s$-factors in $G$ associated with $A$ correspond to 
maximal independent sets in a Cayley graph $\Cay(G, S)$, where $S=A^{-1}A\setminus\{e\}$. 
Consequently, the upper and lower indices of $A$ are equal to the independence number 
and the independent domination number of the associated Cayley graph.

The concepts of $s$-factors, subset indices in groups, stable subsets, and stable groups 
(under different names) were introduced by Hooshmand in 2020. 
Later, Hooshmand and Yousefian-Arani classified stable groups using computer calculations.

Using the connection with graphs, we compute the upper and lower indices 
for various groups and their subsets. Furthermore, 
we prove a classification theorem describing all stable groups 
without relying on computer calculations.

\end{abstract}

\section{Introduction}
\label{section:Introduction}
Let $G$ be a group and let $A$ be a fixed non-empty subset of $G$.
A non-empty subset $B\subset G$ is called a \textit{right $s$-factor of $G$ associated with $A$}
if every element $x\in AB$ can be written uniquely as $x=ab$ with $a\in A$ and $b\in B$
and $B$ is maximal (with respect to inclusion) with this property.
Right $s$-factors exist for every non-empty subset of $G$ (by Zorn’s lemma).

If $G$ is a finite group, then the size of a largest right $s$-factor in $G$ 
associated with $A$ is denoted by $|G:A|^+$, and the size of a smallest one is $|G:A|^-$ 
(these are the upper and lower right indices of $A$, respectively).
A subset $A\subset G$ is called \textit{right stable} in $G$ if its
right upper and lower indices are equal.
A group is called \textit{right stable} if all its subsets are right stable.

Left $s$-factors, left indices, and left stability are defined in the same way.
Throughout the paper, we focus on right $s$-factors and will often omit the qualifier "right".

The concept (under the name "sub-factor") was introduced in \cite{Hooshmand2020} and \cite{Hooshmand2023}. 
The term "subfactor" is well known in the theory of operator algebras, 
particularly in relation to von~Neumann factors; see, e.g., \cite{Jones}. 
Therefore, we use the term "$s$-factor". 
While it may not be ideal, it is shorter and avoids this conflict.
We also replace the terms "sub-index", "index stable set", 
and "index stable group" from \cite{Hooshmand2020} with "index," "stable set", 
and "stable group", respectively.

Note that for any $x\in G$, the sets $A$ and $xA$ have the same collection of right $s$-factors,
and thus $|G:A|^{\pm}=|G:xA|^{\pm}$.
Hence in many arguments we may assume $e\in A$.
Additionally, if $B$ is a right $s$-factor in $G$ associated with $A$,
then for any $y \in G$, $By$ is also a right $s$-factor of $G$ associated with $A$.

\medskip

Let $\Gamma$ be a finite simple graph.
We denote the vertex set and the edge set of $\Gamma$ by $V(\Gamma)$ and $E(\Gamma)$, respectively.
A subset of $V(\Gamma)$ is \textit{independent} if no two of its vertices are adjacent.
An independent set is \textit{maximal} if it is not properly contained in any larger independent set.
The size of a largest independent set in $\Gamma$ is the \textit{independence number} of $\Gamma$.
We denote it by $\alpha(\Gamma)$ (although $\beta(\Gamma)$ or $\beta_0(\Gamma)$ is also used in the literature).
Similarly, the \textit{lower independence number} $i(\Gamma)$ is defined as the size 
of a smallest maximal independent set in $\Gamma$.

An equivalent viewpoint uses dominating sets.
A subset $S$ of the vertices of $\Gamma$ is called \textit{dominating}
if every vertex not in $S$ is adjacent to some vertex in $S$.
The size of a minimum dominating set is the \textit{domination number} $\gamma(\Gamma)$.
The size of a minimum independent dominating set of $\Gamma$ is called 
the \textit{independent domination number} of $\Gamma$.

Since every maximal independent set is a dominating set, 
the lower independence number and the independent domination number coincide,
and, as already indicated, are denoted by $i(\Gamma)$.
As an illustration, consider the following statement.
\begin{lemma}
\label{lemma:Cyclic groups}
Let $C_n$ be the cycle graph on $n$ vertices.
Then $i(C_n)=\lceil n/3\rceil$ and $\alpha(C_n)=\lfloor n/2\rfloor$ for all $n\geq3$.
\end{lemma}
\begin{proof}
These facts are well-known. 
For $i(C_n)$, see \cite[Proposition 1.1]{Goddard}, and for $\alpha(C_n)$, see \cite{Wolfram}.
\end{proof}

The graph $\Gamma$ is  \textit{well-covered} 
if every maximal independent set in $\Gamma$ is maximum,
or equivalently $\alpha(\Gamma)=i(\Gamma)$.
This concept was introduced by Plummer in 1970 \cite{Plummer}.
Simple examples of well-covered graphs are the cycles $C_3$, $C_4$, $C_5$, and $C_7$,
and no other cycle $C_n$ is well-covered.
This follows immediately from Lemma~\ref{lemma:Cyclic groups}.

We recall the definition of a Cayley graph.
We work with left Cayley graphs, defined as follows. 
Let $G$ be a group, and let $S$ be a subset of $G$ such that $e\notin S$ and $S=S^{-1}$.
The (left) Cayley graph $\Cay(G,S)$ is the graph with vertex set $G$
and edge set
\[
\{\{g,sg\}\mid g\in G,s\in S\}.
\]
Vertices $u,v\in G$ are adjacent (sometimes written $u \sim v$) 
if and only if $v u^{-1}\in S$.

Right Cayley graphs are defined analogously.
The vertex set is $G$, and the edges are pairs of vertices 
of the form $\{g,gs\}$, where $g\in G$ and $s\in S$. 
Throughout the paper, unless stated otherwise, $\Cay(G,S)$ refers to the left Cayley graph.

Observe that $\Cay(G,S)$ is connected if and only if $S$ generates $G$.
If $H = \langle S \rangle$ is the subgroup generated by $S$, 
then $\Cay(G,S)$ has $|G:H|$ connected components, 
each isomorphic to $\Cay(H,S)$ and corresponding to a right coset $Hg$ with $g \in G$.
Moreover, $\Cay(G,S)$ is $|S|$-regular.

We are now ready to describe the relationship between $s$-factors and independent sets in graphs.
First, we introduce the notation:
\begin{equation}\label{partial}
\partial A = A^{-1} A \setminus \{e\}.
\end{equation}

\begin{lemma}\label{lemma:The relationship between $s$-factors and Cayley graphs}
  Let $G$ be a group, and let $A$ be a fixed non-empty subset of $G$.
    
  \textup{(i)} A subset $B\subset G$ is a right $s$-factor of $G$ associated with $A$
  if and only if $B$ is a maximal independent set in the Cayley graph $\Cay(G,\partial A)$.
    
  \textup{(ii)} Moreover, we have
\[
|G:A|^+=\alpha\bigl(\Cay(G,\partial A)\bigr),\quad
|G:A|^-=i\bigl(\Cay(G,\partial A)\bigr).
\]
\end{lemma}
\begin{proof}
If $x\neq y$, then
\[
x\sim y\;\Leftrightarrow\; yx^{-1}\in A^{-1}A\; \Leftrightarrow\;Ax\cap Ay\neq\varnothing,
\]
and the claim follows.
\end{proof}

\begin{lemma}[see \cite{HooshmandArani}, Theorem~4.1]
\label{lemma:indices 01 in Zn}
Let $A=\{0,1\}\subset\mathbb{Z}_n$. Then for each $n\geq3$
\[
|\mathbb{Z}_n:A|^-=\lceil n/3\rceil\text{ and }
|\mathbb{Z}_n:A|^+=\lfloor n/2\rfloor.
\]
\end{lemma}
\begin{proof}
It is clear that $\partial A=\{1,-1\}$, and the Cayley graph $\Cay(\mathbb{Z}_n,\partial A)$ 
is the cycle $C_n$.
Hence the lemma follows from Lemmas~\ref{lemma:Cyclic groups} 
and~\ref{lemma:The relationship between $s$-factors and Cayley graphs}.
\end{proof}

The rest of the paper is organized as follows.
In Section~\ref{section:general properties}, we prove a general statement 
about the instability of groups with a cyclic quotient.
In Section~\ref{section:Dihedral}, 
we compute the independence number and the independent domination number
of a Cayley graph of the dihedral group $D_n$ of order $2n$,
with respect to a specific generating set.
In Sections~\ref{section:group of order 21 and 27} --
\ref{section:Order 16}, 
we prove that the following groups are not stable:
a nonabelian group of order $21$ and $UT(3,3)$;
the groups $\mathbb{Z}_2^5$ and $\mathbb{Z}_3^3$;
the groups $A_4$ and $(C_3\times C_3)\rtimes C_2$; 
and certain groups of order $16$.
In Section~\ref{section:Characterise finite stable groups} we prove 
the classification theorem for finite stable groups.
In Section~\ref{section:gap} we provide \GAP code that shifts the burden of computing 
the independence number and the independent domination number for certain groups 
from Sections~\ref{section:group of order 21 and 27}--\ref{section:Order 16} to \GAP\unskip. 
This code can be used for any finite group given by a presentation.

We use the following notations from group theory and graph theory.
The symbol $e$ denotes the identity element of a group.
If $A$ and $B$ are subsets of a group, then
\[
A^{-1}=\{a^{-1}\mid a\in A\}\quad\text{and}\quad
AB=\{ab\mid a\in A,\;b\in B\}.
\]
If $v$ is a vertex of a graph, then $N(v)$ denotes the set of all neighbors of $v$, 
and $N[v] = N(v)\cup\{v\}$.
If $v_1,\ldots,v_k$ (with $k \ge 1$) are vertices of a graph, then
\[
N[v_1,\ldots,v_k]=\bigcup_{i=1}^k N[v_i].
\]
If $u$ and $v$ are vertices of a graph, then we write $u \sim v$ if and only if $u$ and $v$ are adjacent.

\section{Some general properties of stability}
\label{section:general properties}

\begin{lemma}
\label{lemma:On the subgroup index}
Let $H$ be a subgroup of a finite group $G$, and let $S\subset H$ be a symmetric subset 
not containing the identity
(i.e., $S=S^{-1}$ and $e\notin S$).
Then
\begin{align}
  i\bigl(\Cay(G,S)\bigr)=|G:H|\cdot i\bigl(\Cay(H,S)\bigr), \label{eq:On the subgroup index and IDN}\\
 \alpha\bigl(\Cay(G,S)\bigr)=|G:H|\cdot\alpha\bigl(\Cay(H,S)\bigr).\label{eq:On the subgroup index and IN}
\end{align}
\end{lemma}
\begin{proof}
Let $T$ be a right transversal for $H$ in $G$. 
For each $t\in T$, define the graph $\Cay(Ht,S)$ as follows:
its vertices are the elements of the coset $Ht$, and
two vertices $v,w\in Ht$ are adjacent if and only if there exists
$s\in S$ such that $w=sv$ (that is, $wv^{-1}\in S$).
Note that $\Cay(Ht,S)\cong\Cay(H,S)$.
The graph $\Cay(G,S)$ is the disjoint union of the graphs $\Cay(Ht,S)$ for $t\in T$. 
Hence
\eqref{eq:On the subgroup index and IDN} and 
\eqref{eq:On the subgroup index and IN}
both follow.
\end{proof}

\begin{corollary}
\label{corollary: index stable of subgroups}
  If a finite group is stable, then all its subgroups are stable.
\end{corollary}

\begin{lemma}
\label{lemma:On cyclic quotient group}
    Let $H$ be a normal subgroup of a finite group $G$. 
    If the quotient group $G/H$ is a cyclic group of order $n\geq4$, 
    and $|H|\geq3$, then $G$ is not stable.
\end{lemma}
\begin{proof}
Choose an element $g\in G$ such that $\langle Hg\rangle=G/H$.
We can assume that $g^n=e$; otherwise, the order of $g$ is a multiple of $n$ 
and is strictly larger than $n$, hence at least $2n\geq 8$,
and by Lemma~\ref{lemma:indices 01 in Zn} and 
Corollary~\ref{corollary: index stable of subgroups} 
the group $G$ is not stable.
Further, if $\operatorname{ord}(g)=n$ and $n\notin\{4,5,7\}$, then $\langle g\rangle$ is 
not stable by Lemma~\ref{lemma:indices 01 in Zn}, and hence $G$ is not stable by 
Corollary~\ref{corollary: index stable of subgroups}. 
Although we could restrict our attention to $n\in\{4,5,7\}$, 
we carry out the proof for general $n\geq4$, as it does not substantially simplify the argument.

Let $h\in H$ be a nonidentity element, and define
\[
    A=\bigl(H\setminus\{h\}\bigr)\cup \{g\}.
\]
Then
\begin{equation}\label{eq:partial A and A-1A}
  A^{-1}A=H\;\cup\;
    \left(Hg\setminus\{h^{-1}g\}\right)\;\cup\; 
    \left(g^{-1}H\setminus\{g^{-1}h\}\right),
\end{equation}
and $\partial A=A^{-1}A\setminus\{e\}$. 
Here it is important to note that $h\in A^{-1}A$, which is straightforward to verify.
Let $k=\lceil n/2\rceil-1$ and 
let $x_1,\ldots,x_k$ be some elements of $H$
(the precise choice will be specified later). 
For convenience, set $x_0=e$ and define
\[
v_i=g^{2i}x_i,\;i=0,1,\ldots,k.
\]
If $2k=n-1$ (i.e., if $n$ is odd), choose $x_k=h$. In this case $v_k=g^{-1}h$.
We now show that the set
\[
I=\{v_0,v_1,\ldots,v_k\}
\]
is an independent set in the Cayley graph $\Cay(G,\partial A)$
for any choice of $x_i\in H$ for $1\leq i\leq k$ (with $x_k=h$ when $n$ is odd).
Suppose, for contradiction, that two vertices
$v_i$ and $v_j$ (with $0\leq 2i<2j<n$) are adjacent in $\Cay(G,\partial A)$;
that is, 
\[
v_jv_i^{-1}\in\partial A.
\]
There are three cases to consider:
\[
  \text{(i)}\ v_jv_i^{-1}\in H,\quad 
  \text{(ii)}\ v_jv_i^{-1}\in Hg,\quad 
  \text{(iii)}\ v_jv_i^{-1}\in g^{-1}H.
\]

\smallskip\noindent
\textbf{Case} (i).
If $v_jv_i^{-1}\in H$, then 
\[
v_jv_i^{-1}=g^{2(j-i)}(g^{2i}x_jx_i^{-1}g^{-2i})\in H,
\]
so in particular $g^{2(j-i)}\in H$. 
Since $0<2(j-i)<n$, this contradicts the assumption that the order of $g$ modulo $H$ is $n$.

\smallskip\noindent
\textbf{Case} (ii).
If $v_jv_i^{-1}\in Hg$, then a similar calculation shows that
$g^{2(j-i)-1}\in H$.
Since $1\leq2j-2i-1<n$, this case is also impossible.

\smallskip\noindent
\textbf{Case} (iii).
If $v_jv_i^{-1}\in g^{-1}H$, then $g^{2(j-i)+1}\in H$.  
This implies $2(j-i)+1=n$. 
In particular, we must have $j=k$ and $i=0$.
But by our choice (when $2k=n-1$) we set $x_k=h$,
so that
\[
v_jv_i^{-1}=v_k=g^{2k}x_k=g^{n-1}h=g^{-1}h.
\]
Since we assumed $v_jv_i^{-1}\in\partial A$ and $v_jv_i^{-1}=g^{-1}h$, it follows that $g^{-1}h\in\partial A$.
However, from \eqref{eq:partial A and A-1A}, we conclude that
$g^{-1}h\notin\partial A$.
This contradiction implies that this case is also impossible.

Thus, no edge exists between any two distinct vertices in $I$, 
which means that $I$ is indeed an independent set in $\Cay(G,\partial A)$.

\begin{figure}[t]
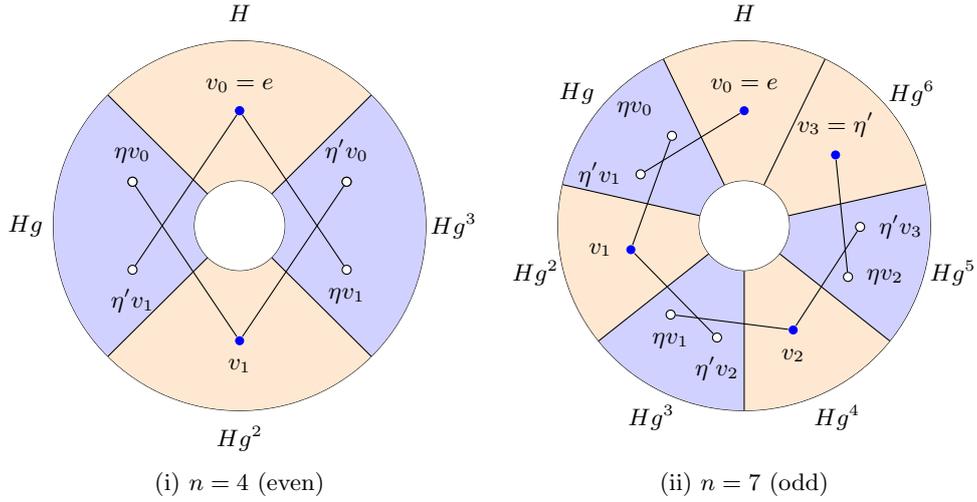

\centering
\def\R{2.45}     
\def\r{0.6}       
\def\RH{\R*1.15}  
\def\CoordLabel{\R*1.4} 

\CyclicQuotientGroupEven{4}
\CyclicQuotientGroupOdd{7}

\caption{\footnotesize
For Lemma~\ref{lemma:On cyclic quotient group}.
The set $I=\{v_0,\dots,v_k\}$ lies in the even cosets.
For each $i$, $N[v_i]=(A^{-1}A)v_i$ covers $Hg^{2i}$ and covers each adjacent odd coset $Hg^{2i\pm1}$ except for two “holes”: $\eta v_i$ and $\eta'v_i$.
The condition $\eta v_i\neq \eta'v_{i+1}$ ensures $Hg^{2i+1}\subset N[v_i]\cup N[v_{i+1}]$.
In the odd case, $v_k=g^{-1}h\in Hg^{n-1}$ and the unique non-neighbor of $v_k$ in $Hg^{2k-1}$ is $(g^{-1}h)^2$. 
}
\label{fig:coset-cycle}
\end{figure}

Next, we prove that with an appropriate choice of $x_i$, 
the set $I$ is a dominating set in $\Cay(G,\partial A)$.
It is convenient to follow the subsequent argument using Figure~\ref{fig:coset-cycle}.
By the definition of the Cayley graph $\Cay(G, \partial A)$, we have 
$N[v_i] = (A^{-1}A)v_i$ 
for each $i = 0, 1, \ldots, k$. From \eqref{eq:partial A and A-1A}, 
it follows that for each $i=0,1,\ldots,k$, the following inclusions hold:
\begin{align}
  & Hg^{2i}\subset N[v_i], \label{eq:Inclusions even cosets}\\
  & Hg^{2i+1}\setminus\{h^{-1}gv_i\}=
    Hg^{2i+1}-\eta v_i\subset N[v_i],\label{eq:Inclusions odd cosets above}\\
  & Hg^{2i-1}\setminus\{g^{-1}hv_i\}=
    Hg^{2i-1}-\eta' v_i\subset N[v_i],\label{eq:Inclusions odd cosets below}
\end{align}
where, for the sake of compactness, we write $X - x$ instead of $X\setminus \{x\}$ 
and define $\eta = h^{-1}g$ and $\eta' = \eta^{-1}$.
We list the cosets appearing in formulas \eqref{eq:Inclusions odd cosets above} and 
\eqref{eq:Inclusions odd cosets below} in the table below:
\begin{equation}\label{array:Odd cosets}
\setlength{\arraycolsep}{2pt} 
\begin{array}{cccccc}
                     & Hg-\eta v_0, &  \ldots, & Hg^{2k-1}-\eta v_{k-1}, & Hg^{2k+1}-\eta v_k, \\
  Hg^{n-1}-\eta'v_0, & Hg-\eta'v_1, &  \ldots, & Hg^{2k-1}-\eta'v_k,     & 
\end{array}
\end{equation}
From \eqref{array:Odd cosets}, we see that if $\eta v_i \neq \eta' v_{i+1}$, then 
\begin{equation}\label{eq:Inclusions odd cosets}
Hg^{2i+1} \subset N[v_i] \cup N[v_{i+1}]
\end{equation}
for each $i = 0, \ldots, k-1$. 
To ensure the condition $\eta v_i \neq \eta' v_{i+1}$ for $i = 0, \ldots, k-2$, 
it is clear how $x_1, \ldots, x_{k-1}$ should be chosen.
If $x_1,\ldots, x_i$ have already been chosen, we select $x_{i+1}$ so that 
$\eta v_i \neq \eta' v_{i+1}$, which is equivalent to the condition 
\[
x_{i+1} \neq g^{-2i-2}(h^{-1}g)^2g^{2i}x_i,\;i = 0, \ldots, k-2.
\]
Recall that, by convention, $x_0 = e$.
Finally, let us consider the case when $i = k - 1$, that is, when $x_k$ has been chosen.
We distinguish two cases: $n$ is even or $n$ is odd.

If $n$ is even, we must choose $x_k$ such that 
$\eta v_{k-1} \neq \eta' v_k$ and, since in this case 
$2k+1 = n-1$, $\eta v_k \neq \eta' v_0$.
Thus, we choose $x_k$ such that 
\[
x_k\neq g^{-2k}(h^{-1}g)^2g^{2k-2}x_{k-1} \text{ and } x_k\neq g^{-2k}(g^{-1}h)^2.
\]
Since $|H| \geq 3$, both conditions can be satisfied.

Now, let $n$ be odd. 
We know that in this case, $n=2k+1$, $v_k=g^{-1}h$,
and $\left(g^{-1}h\right)^2$ is the only vertex in
the coset $Hg^{2k-1}$ that is not adjacent to $v_k$.
Since $x_k = h$ is fixed, we adjust the choice of $x_{k-1}$ (already made) 
to ensure that both $\eta' v_{k-1} \neq \eta v_{k-2}$ and $\eta v_{k-1} \neq \eta' v_k$ hold.
We must satisfy the condition
\[
x_{k-1}\neq g^{-2k+2}(h^{-1}g)^2g^{2k-4}x_{k-2}\text{ and } 
x_{k-1}\neq g^{-2k+2}(g^{-1}h)^3.
\]
As above, such a choice is possible since $|H| \geq 3$.

Thus, by \eqref{eq:Inclusions even cosets}, every coset of the form $Hg^{2i}$ 
lies in $N[v_i]$ for $i = 0, 1, \ldots, k$.
Moreover, with an appropriate choice of $x_i$, it follows from \eqref{eq:Inclusions odd cosets} that 
every coset of the form $Hg^{2i+1}$ lies in $N[v_i] \cup N[v_{i+1}]$ for $i = 0, 1, \ldots, k-1$. 
As a result 
\[
G=\bigcup_{i=0}^k N[v_i].
\] 
Therefore, $I$ is a dominating set in $\Cay(G, \partial A)$.
Thus, we have proved that $I$ is an independent dominating set. 
Hence, $i\bigl(\Cay(G,\partial A)\bigr)\leq \lceil n/2\rceil$.

On the other hand, the set 
\[
\{e, h^{-1}g,(h^{-1}g)^2,\ldots,(h^{-1}g)^{n-2}\}
\] 
is independent, 
since for any $0\leq i<j\leq n-2$ we have $(h^{-1}g)^{j-i}\notin\partial A$, 
and hence $(h^{-1}g)^j$ is not adjacent to $(h^{-1}g)^i$ in $\Cay(G,\partial A)$.
Consequently, the size of the largest independent set in $\Cay(G,\partial A)$ is at least $n-1$, that is, $\alpha\bigl(\Cay(G,\partial A)\bigr)\geq n-1$.
Since $n-1>\lceil n/2\rceil$ for all $n>3$,  we conclude that 
\begin{equation*}
    \alpha\bigl(\Cay(G,\partial A)\bigr) > i\bigl(\Cay(G,\partial A)\bigr). \qedhere
\end{equation*}
\end{proof}

To conclude this section, we present a general statement 
that reduces the computation of the independent domination number and the independence number 
to graphs of smaller order.

\begin{lemma}\label{lemma:vertex-transitive} 
Let $\Gamma$ be a vertex-transitive graph, 
let $v$ be an arbitrary vertex, 
and let $X$ be the induced subgraph on $V(\Gamma)\setminus N[v]$.
Then 
\[
i(\Gamma) = i(X) + 1 \quad \text{and} \quad \alpha(\Gamma) = \alpha(X) + 1.
\]
In particular, these formulas hold for any Cayley graph.
\end{lemma} 

\begin{proof}
A graph is vertex-transitive if, for any two vertices $x$ and $y$, 
there exists an automorphism mapping $x$ to $y$. 
Let $I$ be a maximal independent set, and let $x\in I$. 
Then there exists an automorphism $f$ of $\Gamma$ such that $f(x)=v$.
Consequently, $f(I)$ is also a maximal independent set and $|f(I)|=|I|$.
Thus, without loss of generality, we may assume that $v\in I$.
Furthermore, if $I'=I\setminus\{v\}$, then $I'\cap N[v] = \varnothing$, 
and hence $I' \subset V(X)$. Moreover, $I'$ is a maximal independent set in $X$.
This proves the claimed equalities.
Indeed, if $I$ is a smallest maximal independent set, then 
\[
i(\Gamma) = |I| = |I'| + 1 = i(X) + 1.
\] 
Similarly, we obtain the equality $\alpha(\Gamma) = \alpha(X) + 1$.
The ``in particular'' statement follows because every Cayley graph is vertex-transitive: 
right multiplication by a group element is an automorphism (for left Cayley graphs).
\end{proof}

\section{Dihedral groups}
\label{section:Dihedral}

Here, we compute the independence number and independent domination number 
for certain Cayley graphs of dihedral groups $D_n$, where $n\geq3$.
Recall that $D_n$ has order $2n$ and can be defined by the presentation
\[
D _n=\left\langle a,b\mid a^2=e,b^2=e,(ab)^{n}=e\right\rangle.
\]
A few words about notation. Since in Sections~\ref{section:Order 16} and~\ref{section:gap}
we will refer to the computer algebra system \GAP \cite{GAP}, 
we briefly explain the \GAP notation we use.
In \GAP\unskip, the dihedral group of order $m$ (with $m$ even) 
is constructed as \verb"DihedralGroup(m)",
and its structure is typically reported as the string \verb"Dm" (e.g., \verb"D8" or \verb"D16").
Throughout this note, however, $D_n$ denotes the dihedral group of order $2n$.
Thus, \verb"DihedralGroup(2n)" in \GAP corresponds to our $D_n$
(e.g., \verb"D8" corresponds to $D_4$, and \verb"D16" corresponds to $D_8$).

\begin{lemma}
\label{lemma:indices Dn}
Let $D_n=\langle a,b\rangle$ be a dihedral group, where $a^2=b^2=(ab)^n=e$ and $n\geq3$.
If $S=\partial\{e,a,b\}$, then
\begin{align}
  i\bigl(\Cay(D_n,S)\bigr)=\lceil2n/5\rceil,          \label{eq:IDN dihedral group}   \\
  \alpha\bigl(\Cay(D_n,S)\bigr)=\lfloor2n/3\rfloor.   \label{eq:IN dihedral group}
\end{align}
\end{lemma}

\begin{proof}
It is straightforward to check that $S = \{a,b,ab,ba\}$. 
We also use the following known inequality 
(see, e.g., \cite[p.~278]{Berge}), 
which holds for any simple graph $\Gamma$ with maximum degree $\Delta$:
\begin{equation}\label{ineq:IDN Berge}
    i(\Gamma)\geq\left\lceil\frac{|\Gamma|}{\Delta+1}\right\rceil.
\end{equation}
In our case, $|\Gamma|=2n$ and $\Delta=4$.
Hence, to prove \eqref{eq:IDN dihedral group}, 
it suffices to construct a maximal independent set of size $\lceil 2n/5 \rceil$.

Observe that the Cayley graph $\Cay(D_n,S)$ has a Hamiltonian cycle
\[
P = (e,\, b,\, ab,\, bab,\, \ldots,\, (ab)^{n-1},\, b(ab)^{n-1},\,e).
\]
Label the vertices by the integers $0,1,\ldots,2n-1$ in the order in which they appear on this cycle. 
We will identify each vertex with its label. 
The graph has two additional cycles of length $n$, namely
\[
e,\ ab,\ (ab)^2,\ \ldots,\ (ab)^{n-1},\ e
\quad\text{and}\quad
b,\ (ba)b,\ (ba)^2b,\ \ldots,\ (ba)^{n-1}b\ b.
\]
The vertices of the first cycle have even labels along $P$, 
while the vertices of the second cycle have odd labels.
Moreover, successive vertices in the first cycle are obtained by left multiplication by $ab\in S$, 
and successive vertices in the second cycle by left multiplication by $ba\in S$.
Since $\Cay(D_n,S)$ is $4$-regular (as $|S|=4$), 
and with the above labeling along $P$,
it follows that two distinct vertices 
$i$ and $j$ are adjacent if and only if 
their distance along the cycle $P$ is at most $2$, or equivalently,
\[
i-j\equiv\pm1\smod{2n}\quad\text{or}\quad i-j\equiv\pm2\smod{2n}.
\]
Figure~\ref{fig:Cayley graphs of dihedral groups} shows the graphs 
$\Cay(D_4,S)$, $\Cay(D_5,S)$, and $\Cay(D_6,S)$.

\begin{figure}[ht]
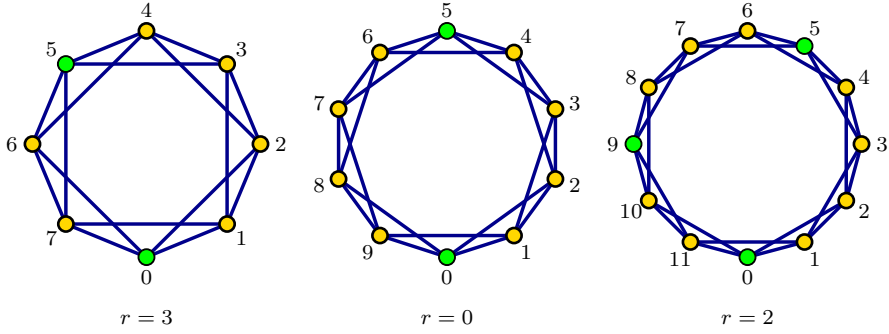

    \centering
    \DihedralGraph{4}{15mm}{1mm}    
    \DihedralGraph{5}{15mm}{1mm}    
    \DihedralGraph{6}{15mm}{1mm}
    \caption{Graphs $\Cay(D_4,S)$, $\Cay(D_5,S)$, $\Cay(D_6,S)$}
    \label{fig:Cayley graphs of dihedral groups}
\end{figure}
Write $2n = 5q + r$, where $q$ and $r$ are integers with $0 \le r \le 4$. 
Depending on the value of $r$, define the set $I$ as follows:
\begin{equation*}
\renewcommand{\arraystretch}{1.3}
I=
\left\{
  \begin{array}{ll}
    \{5k\mid k=0,1,\ldots,q-1\},             & \text{if}\ r=0; \\
    \{5k\mid k=0,1,\ldots,q-1\}\cup\{2n-3\}, & \text{if}\ r=1,2;\\
    \{5k\mid k=0,1,\ldots,q\},               & \text{if}\ r=3,4.
  \end{array}
\right.
\end{equation*}

\smallskip\noindent
\textbf{Claim 1:} $I$ is an independent set.  
Indeed, suppose $5i \sim 5j$. Then
\begin{equation}\label{eq:5j-5i equiv 1 or 2 mod 2n}
    5j-5i\equiv\pm1,\pm2\smod{2n}. 
\end{equation}
Without loss of generality, assume $q\geq j>i\geq0$. 
Then $0<5(j-i)\leq2n$. 
By the definition of $I$, we have $j-i\leq q$ (and in fact $j-i\leq q-1$ unless $r\geq3$), hence $5(j-i)\neq 2n$;
and if $5(j-i)<2n$, then $5(j-i)$ cannot be $\pm1$ or $\pm2$.

It remains to verify that for $r\in\{1,2\}$ the vertex $2n-3$ is not adjacent to any $5i$ 
with $0\leq i<q$. Suppose, to the contrary, that $2n-3\sim5i$. Then
\[
(2n-3)-5i=\pm1,\pm2\smod{2n}.
\]
Since $0<(2n-3)-5i<2n$, we get 
\[
  2n-3-5i=5(q-i)+r-3=1\ \text{or}\ 2.
\]
This implies $r\equiv4\smod{5}$ or $r\equiv0\smod{5}$, 
contradicting $r\in\{1,2\}$. Hence $I$ is independent in all cases.

\smallskip\noindent
\textbf{Claim 2:} $I$ is a maximal independent set.
Let $m$ be any integer with $0 \leq m < 2n$. 
If there exists $i$ (where $0 \leq i < q$) such that $5i < m < 5(i+1)$, 
then either $0 < 5(i+1) - m \leq 2$ or $0 < m - 5i \leq 2$, 
implying that $m \sim 5(i+1)$ or $m \sim 5i$. 
If no such $i$ exists, then $m > 5(q-1)$. 
We now consider the three cases:
\begin{itemize}[nosep]
  \item If $r=0$, then $m\sim5(q-1)$ or $m\sim0$;         
  \item If $r=3,\;4$, then $m\sim5(q-1)$, or $m\sim5q$, or $m\sim0$;  
  \item If $r=1,\;2$, then $m\sim5(q-1)$, or $m\sim2n-3$, or $m\sim0$.
\end{itemize}
(In Figure~\ref{fig:Cayley graphs of dihedral groups}, the vertices in $I$ are shown in green.)  
Thus, in all cases, $I$ is a maximal independent set and
$|I|=\lceil2n/5\rceil$.
This proves \eqref{eq:IDN dihedral group}.

We now prove \eqref{eq:IN dihedral group}. 
We have seen that if two vertices of $\Cay(D_n,S)$ are not adjacent, 
then their distance along the Hamiltonian cycle $P$ is at least $3$. 
Write $2n = 3q + r$ with $0 \leq r \leq 2$.
Assume for a contradiction that there exists an independent set of size $q+1$.
Let
$
I=\{0,i_1,\ldots,i_q\}
$
be an independent set of vertices in $\{0,1,2,\ldots,2n-1\}$
with $0<i_1<i_2<\ldots<i_q$.
Since $0$ and $2n$ represent the same vertex, we obtain the chain of inequalities
\begin{align*}
i_1\geq3,\\
i_2-i_1\geq3,\\
\ldots,\\
i_{q-1}-i_{q-2}\geq3,\\
i_q-i_{q-1}\geq3,\\
2n-i_q\geq3.
\end{align*}
Summing these gives $2n \geq 3q + 3$, which contradicts $2n = 3q + r$ with $r \leq 2$. 
Hence every independent set has at most $q = \lfloor 2n/3 \rfloor$ vertices. 
On the other hand, it is clear that
$\{0,3,\ldots,3(q-1)\}$ is an independent set of size $q$. 
This establishes \eqref{eq:IN dihedral group}.
\end{proof}

\section{Non-abelian groups of orders 21 and 27}
\label{section:group of order 21 and 27}

This section determines the independence number and the independent 
domination number for selected Cayley graphs of the non-abelian groups 
$C_7\rtimes C_3$ and $UT(3,\mathbb{F}_3)$.

\begin{lemma}
\label{lemma:indices of group of order 21}
Let
\[
G=\langle a,b\mid a^7=b^3=e,\;bab^{-1}=a^2\rangle\quad\text{and}\quad 
S=\partial\{e,a,b\}.
\]
Then
\begin{align}
  i\bigl(\Cay(G,S)\bigr)=3,                   \label{eq:IDN 21}   \\
  \alpha\bigl(\Cay(G,S)\bigr)=6.              \label{eq:IN 21}
\end{align}
\end{lemma}
\begin{proof}
Note that $G$ is a semidirect product of a cyclic group of order $7$ 
by a cyclic group of order $3$.
Since 
$S=\{a,a^{-1},b,b^{-1},a^{-1}b,b^{-1}a\}$ and $|S|=6$,
applying \eqref{ineq:IDN Berge} yields $i\bigl(\Cay(G,S)\bigr)\geq3$.
On the other hand, consider the set
$
    I=\{e,ba^5,b^2a^4\}.
$
It is easy to check that $I$ is a maximal independent set of size three.
Hence, $i\bigl(\Cay(G,S)\bigr)=3$.

\begin{figure}[ht]
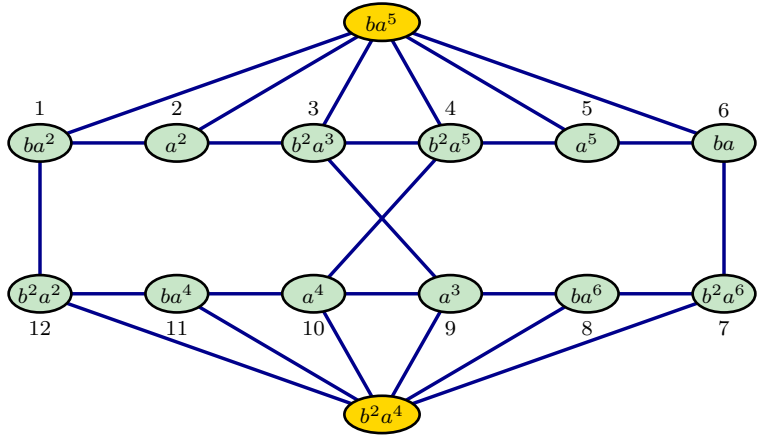

    \centering
    \CayLeyGraphTwentyOne{21}{18mm}{3.3mm}{}
    \caption
    {The subgraph of $\Cay(G,S)$ with vertex set $G\setminus N[e]$, $|G|=21$}
    \label{fig:Cayley Graph Twenty-One}
\end{figure}

We now prove \eqref{eq:IN 21}.
Since $a^{-1}=a^6$, $b^{-1}=b^2$, and
\[
a^{2i}b=ba^i,\;a^{4i}b^2=b^2a^i,\;i=1,\ldots,6,
\]
we have
\[
  N(e)=\{a,\; a^6,\; b,\; ba^3,\; b^2,\; b^2a\}.  
\]
We have listed these elements so that those from the subgroup 
$H=\langle a\rangle$ appear first, followed by those from $bH$ and then $b^2H$.

Let $X$ be the subgraph of $\Cay(G,S)$ with vertex set $G\setminus N[e]$
and edge set consisting of the edges shown in Figure~\ref{fig:Cayley Graph Twenty-One}.  
It contains a cycle $P$ of length $12$, whose vertices are colored green there.  
If $X$ had an independent set $J$ of size $6$, then $J\subset V(P)$, 
implying that every second vertex on $P$ belongs to $J$.
By examining the two cases 
(where, under a labeling of the vertices of $P$, the vertices in $J$ are 
either all even-numbered or all odd-numbered) 
in Figure~\ref{fig:Cayley Graph Twenty-One},
one sees that this is impossible. Thus, $\alpha(X)\leq5$. 
Note that if $Y$ is the induced subgraph on $G\setminus N[e]$, then $X$ is a subgraph of $Y$ and
hence $\alpha(Y)\leq\alpha(X)\leq5$.
Therefore, Lemma~\ref{lemma:vertex-transitive} implies that $\alpha\bigl(\Cay(G,S)\bigr)\leq6$.

On the other hand, $\Cay(G,S)$ does have an independent set of six vertices: 
\[
\{e,\; a^4,\; ba,\; ba^6,\; b^2a^2,\; b^2a^3\}.
\]
We conclude that $\alpha\bigl(\Cay(G,S)\bigr) = 6$.
\end{proof}

\begin{lemma}
\label{lemma:IDN and IN of group of order 27}
Let
\[
G=\langle a,\; b,\; c\mid a^3=b^3=c^3=e,\; b^{-1}a^{-1}ba=c,\;ac=ca,\; bc=cb\rangle.
\]
If $S=\partial\{e,a,b,c\}$,
then
\begin{align}
  i\bigl(\Cay(G,S)\bigr)=3,                   \label{eq:IDN 27-ii}   \\
  \alpha\bigl(\Cay(G,S)\bigr)=6.              \label{eq:IN 27-ii}
\end{align}
\end{lemma}

\begin{proof}
Although we do not use this fact, it is useful to note that
$G$ is isomorphic to the group of unitriangular matrices $UT(3,\mathbb{F}_3)$ 
over the field $\mathbb{F}_3$ with three elements.
The mapping defined by 
\[
a\mapsto t_{1,2}(1),\;
b\mapsto t_{2,3}(1), \text{ and } 
c \mapsto t_{1,3}(-1) 
\]
extends to an isomorphism $G\to UT(3,\mathbb{F}_3)$.
Here, $t_{i,j}(\sigma)$ denotes a transvection, that is
the identity matrix with its $(i,j)$-entry replaced by $\sigma$.

\begin{figure}[ht]
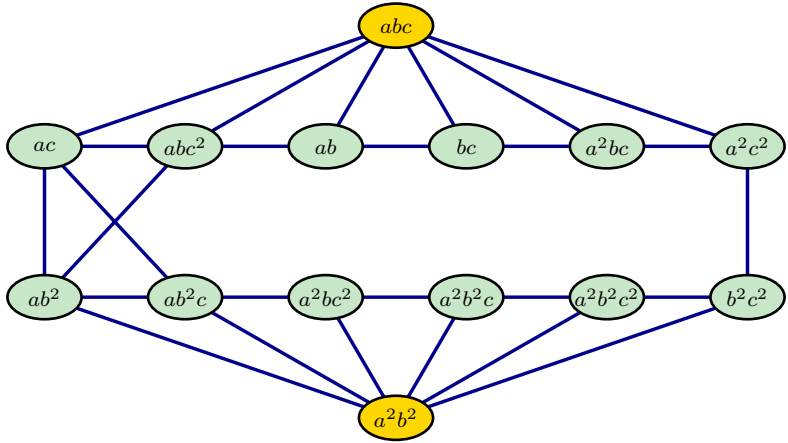

    \centering
    \SubgraphCayLeyGraphTwentySeven{14}{18.5mm}{3.9mm}{}
    \caption    
    {The subgraph of $\Cay(G,S)$ with vertex set $W=G\setminus N[e]$, where $|G|=27$}
    \label{fig:SubgraphCayley GraphTwentySeven ii}
\end{figure}
Using $ba=abc$ and the fact that $c\in Z(G)$, we obtain
\[
S=\{a,\; b,\;  c,\; a^2,\; b^2,\; c^2,\; a^2b,\; ab^2c^2,\; a^2c,\; ac^2,\; b^2c,\; bc^2\}.
\]
Let $W$ be the subgraph of $\Cay(G,S)$ with vertex set $G\setminus N[e]$ and 
edge set consisting of the edges shown
in Figure~\ref{fig:SubgraphCayley GraphTwentySeven ii}.
To verify that two vertices $u,v\in W$ are adjacent, 
it suffices to check that $uv^{-1} \in S$.  
For example, since 
\[
(abc)(a^{2}c^{2})^{-1} = (abc)(ca)=abac^{2} = a^{2}bc^3=a^2b \in S,
\] 
we have $abc \sim a^{2}c^{2}$.

Applying the bound \eqref{ineq:IDN Berge} with $\Delta=|S|=12$, we obtain
\[
i\bigl(\Cay(G,S)\bigr)\geq \left\lceil\frac{27}{12+1}\right\rceil = 3.
\]
One can verify that the set $D=\{e,\,abc,\,a^2b^2\}$ is independent in $\Cay(G,S)$.
Moreover, $abc$ and $a^2b^2$ dominate every vertex of $W=G\setminus N[e]$ in $\Cay(G,S)$, 
while $e$ dominates $N[e]$.
Hence, $D$ is an independent dominating set in $\Cay(G,S)$, and therefore
$i\bigl(\Cay(G,S)\bigr)\le 3$.
Combined with the lower bound above, this yields $i\bigl(\Cay(G,S)\bigr)=3$.

Next, it is also straightforward to check that the set of six vertices 
\[
\{e,\; ac,\; ab,\; a^2bc,\; b^2c^2,\; a^2b^2c\}
\] 
is independent in $\Cay(G,S)$.  
On the other hand, the subgraph $W$ does not contain an independent set of size $6$. 
Indeed, any such set would lie entirely on the $12$-cycle
(whose vertices are colored green in Figure~\ref{fig:SubgraphCayley GraphTwentySeven ii}),
but the two additional edges $\{ac, ab^2c\}$ and $\{ab^2, abc^2\}$ 
prevent an independent set of size $6$ on that cycle 
(which would have to consist of alternating vertices). 
By Lemma~\ref{lemma:vertex-transitive}, we therefore conclude that 
$\alpha\bigl(\Cay(G,S)\bigr) = 6$.
\end{proof}

\section{The groups $\mathbb{Z}_2^5$ and $\mathbb{Z}_3^3$}
\label{section:Direct}

In this section, we compute the independence number and the independent domination number 
for certain Cayley graphs of the elementary abelian groups 
$\mathbb{Z}_2^5$ and $\mathbb{Z}_3^3$.

\begin{lemma}[see \cite{HooshmandArani}]
\label{lemma:Z2^5}    
  In the group $\mathbb{Z}_2^5$, choose the set 
  \[
    A=\{0,e_1,e_2,e_3,e_4,e_5\},
  \]
  where $e_i$ is the unit vector of length $5$ 
  having a $1$ in the $i$-th position and $0$s elsewhere.    
  Then
  \[
  S=\{e_i \mid 1\leq i\leq 5\}\cup\{e_i+e_j \mid 1\leq i<j\leq5\},
  \]
  and
\begin{equation}\label{eq:IDN and IN Z2^5}
  i\bigl(\Cay(\mathbb{Z}_2^5,S)\bigr)=2, \quad \alpha\bigl(\Cay(\mathbb{Z}_2^5,S)\bigr)= 4.
\end{equation}  
\end{lemma}

\begin{proof}
Let $G=\Cay(\mathbb{Z}_2^5,S)$.  Write each $x\in\mathbb{Z}_2^5$ uniquely as
$x=\sum_{i=1}^5 x_i e_i$ with $x_i\in\{0,1\}$, and define the Hamming weight
\[
\wt(x)=|\{i \mid x_i=1\}|.
\]
Since $S$ consists of all vectors of weight $1$ or $2$, two vertices $x,y$ are adjacent in $G$
if and only if $\wt(x+y)\in\{1,2\}$.

\smallskip
\noindent\textit{Independent domination.}
Since $|V(G)|=32$ and $\Delta=|S|=15$, inequality~\eqref{ineq:IDN Berge} gives
\[
i(G)\geq \left\lceil\frac{32}{15+1}\right\rceil=2.
\]
Let $u=e_1+e_2+e_3+e_4+e_5$ (so $\wt(u)=5$).
For any $v\in\mathbb{Z}_2^5$, either $\wt(v)\le 2$ or $\wt(u+v)=5-\wt(v)\le 2$.
Hence every vertex lies in $N[0]\cup N[u]$, so $\{0,u\}$ is a dominating set.
Therefore $i(G)\le 2$, and thus $i(G)=2$.

\smallskip
\noindent\textit{Independence number.}
Consider the set
\[
I=\{0,\ e_1+e_2+e_3,\ e_1+e_4+e_5,\ e_2+e_3+e_4+e_5\}.
\]
For any distinct $x,y\in I$, one checks that $\wt(x+y)\ge 3$, 
hence $x+y\notin S$ and $I$ is independent.
Thus $\alpha(G)\geq4$.

To prove that $\alpha(G)\leq4$, let $J$ be a maximum independent set in $G$.
By translation, we may assume that $0\in J$.
Then $\wt(x)\geq3$ for all $x\in J\setminus\{0\}$.

The unique vector of weight $5$ cannot belong to $J$ together with any other nonzero vertex,
since it is adjacent to every vector of weight $3$ or $4$.
Also, $J$ contains at most one vector of weight $4$, because any two distinct weight-$4$ vectors
differ in exactly two coordinates and hence are adjacent.

Finally, $J$ contains at most two vectors of weight $3$.
Indeed, if $A,B,C\subset\{1,\ldots,5\}$ are supports of three weight-$3$ vectors in $J$, then
independence implies $|A\cap B|\leq1$, $|A\cap C|\leq1$, $|B\cap C|\leq1$, 
and hence
\[
|A\cup B\cup C|
=9-|A\cap B|-|A\cap C|-|B\cap C|+|A\cap B\cap C|
\geq 9-3>5,
\]
a contradiction.

Thus $|J|\leq1+1+2=4$, so $\alpha(G)\leq4$, and therefore $\alpha(G)=4$.
\end{proof}

\begin{lemma}
\label{lemma:Z3^3}    
In the group $\mathbb{Z}_3^3$, define $A=\{0,e_1,e_2,e_3\}$, 
where $e_i$ denotes the $i$th standard basis vector. 
Then we have
\[
  S=\{\pm e_i,\ \pm(e_i-e_j)\mid 1\leq i\neq j\leq 3\}
\]
and  
\begin{equation}\label{eq:IDN and IN Z3^3}
  i\bigl(\Cay(\mathbb{Z}_3^3,S)\bigr)=3,\quad 
  \alpha\bigl(\Cay(\mathbb{Z}_3^3,S)\bigr)=4.
\end{equation}
\end{lemma}

\begin{proof}
Represent each element of the group $G=\mathbb{Z}_3^3$ as a string of length three;  
thus, instead of writing $(i,j,k)$, we write $ijk$ for brevity.
Let $\Gamma=\Cay(G,S)$.

Since $|S|=12$, inequality~\eqref{ineq:IDN Berge} yields
\[
i(\Gamma)\geq\left\lceil\frac{27}{12+1}\right\rceil=\left\lceil\frac{27}{13}\right\rceil=3.
\]
Moreover,
\[
D=\{000,111,222\}
\]
is an independent dominating set in $\Gamma$, which implies $i(\Gamma)=3$.
To see that $D$ dominates, observe that if $g\in G$ has three distinct coordinates, 
then $g$ is adjacent to $000$.
On the other hand, $g$ is adjacent to $kkk$ whenever $g$ has at least two coordinates equal to $k$.

\medskip
We now prove that $\alpha(\Gamma)=4$.
It is straightforward to find an independent set of four vertices, 
for example $\{000,\,110,\,211,\,022\}$,
so $\alpha(\Gamma)\ge 4$.

Let $X$ be the induced subgraph of $\Gamma$ on $V(\Gamma)\setminus N[000]$.
By Lemma~\ref{lemma:vertex-transitive}, $\alpha(\Gamma)=\alpha(X)+1$.
Thus it suffices to show that $\alpha(X)\leq3$.

The vertex set of $X$ consists precisely of the vertices that are not adjacent to $000$.
In particular, $111,222\in V(X)$.
Arrange the remaining vertices of $X$, distinct from $111$ and $222$, as follows:
\begin{align*}
A=\{110,101,011\},\; C=\{220,202,022\},\\
B=\{112,121,211\},\; D=\{221,212,122\}.
\end{align*}
Let $Q=\{111,222\}$.
Each of the sets $A,B,C,D$ induces a triangle in $X$.
Moreover, every vertex in $A\cup B$ is adjacent to $111$, 
and every vertex in $C\cup D$ is adjacent to $222$.

Let $J$ be an independent set in $X$.
We claim that $|J|\leq3$.

\smallskip\noindent
\textbf{Case 1: $J\cap Q\neq\varnothing$.}
By symmetry (using $x\mapsto -x$), assume $111\in J$.
Then $J$ contains no vertices from $A\cup B$, and it meets each of $C$ and $D$ 
in at most one vertex.
Hence $|J|\leq1+1+1=3$.

\smallskip\noindent
\textbf{Case 2: $J\cap Q=\varnothing$.}
If $J\cap A=\varnothing$, then $J$ meets each of $B,C,D$ in at most one vertex, 
so $|J|\leq3$.
Otherwise, by symmetry we may assume $110\in J$.
Only three vertices in $B\cup D$ are non-adjacent to $110$, namely
\[
E=\{121,211,221\}.
\]
The set $E$ forms a triangle in $X$, so $|J\cap E|\leq1$.
Together with $|J\cap C|\leq1$, this gives $|J|\leq1+1+1=3$.

Thus $\alpha(X)\leq3$, and therefore $\alpha(\Gamma)=\alpha(X)+1\le 4$.
Consequently, $\alpha(\Gamma)=4$.
\end{proof}

\section{The groups $A_4$ and $(C_3\times C_3)\rtimes C_2$}
\label{section:Order 12 and 18}

In this section, we compute the independence number and the independent domination number 
for certain Cayley graphs of the alternating group $A_4$ and the semidirect product 
$(C_3 \times C_3) \rtimes C_2$, where the nontrivial element of $C_2$ acts on 
$C_3 \times C_3$ by inversion.
We start with a presentation of $A_4$.
\begin{lemma}
\label{lemma:semidirect Z2^2:Z3}
Let 
\[
G=\langle a,b,t\mid a^2=b^2=t^3=e,\;ab=ba,\;t^{-1}at=b,\;t^{-1}bt=ab\rangle
\]
and let $A=\{e,b,t\}$. Then
\[
i\bigl(\Cay(G,\partial A)\bigr)=2\quad \text{and}\quad
\alpha\bigl(\Cay(G,\partial A)\bigr)=3.
\]
\end{lemma}

\begin{proof}
It is easy to verify that $G\cong A_4$. 
Every element of $G$ has a unique normal form $t^i a^j b^k$, where $i\in\{0,1,2\}$ and 
$j,k\in\{0,1\}$. Since $bt=tab$, we obtain
\[
\partial A=(A^{-1}A)\setminus\{e\}=\{b,t,t^2,tab,t^2b\}.
\] 
%
\begin{figure}[ht]
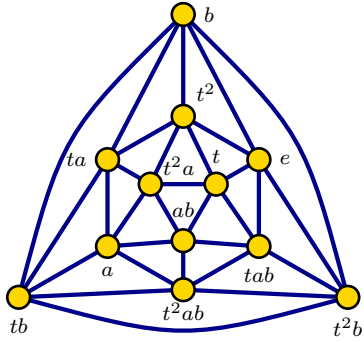

    \centering
    \def\a{5mm} 
    \def\r{1.5mm} 
    \CayleyGraphAlt
\caption
{The Cayley graph $\Cay(A_4,\partial A)$ is the icosahedral graph}
\label{fig:Cayley graph A4}
\end{figure}
%
A short computation shows that $\Cay(G,\partial A)$ is isomorphic to the icosahedral graph, 
which is depicted in Figure~\ref{fig:Cayley graph A4}. 
We observe that $\{e,a\}$ is an independent dominating set, 
so $i\bigl(\Cay(G,\partial A)\bigr)=2$. 
On the other hand, the icosahedral graph has independence number $3$ 
(for example, $\{e,ab,tb\}$ is an independent set of size $3$).
Therefore, $\alpha\bigl(\Cay(G,\partial A)\bigr)=3$.
\end{proof}

\begin{lemma}
\label{lemma:semidirect Z3^2:Z2}
Let 
\[
G=\langle a,b,t\mid a^3=b^3=t^2=e,\;ab=ba,\;tat=a^2,\;tbt=b^2\rangle
\]
and let $A=\{e,a,t,bt\}$.
Then
\[
i\bigl(\Cay(G,\partial A)\bigr)=2,\quad\text{and}\quad
\alpha\bigl(\Cay(G,\partial A)\bigr)=4.
\]
\end{lemma}

\begin{proof}
The group $G$ is the semidirect product of the elementary abelian group 
$\langle a,b\rangle$ of order $9$ and the cyclic group $\langle t\rangle$ of order $2$.
Every element of $G$ can be uniquely written in the form $a^i b^j t^k$, 
with $i, j\in\{0, 1, 2\}$ and $k \in \{0, 1\}$.  
We will use this representation throughout.  
A straightforward computation shows that
\[ 
\partial A = \{a,\; a^2,\; b,\; b^2,\; t,\; bt,\; a^2t,\; a^2bt\}.
\]
One can verify that the closed neighborhoods $N[e]$ and $N[ab^2t]$ cover $G$
and that $\{e,ab^2t\}$ is an independent set in $\Cay(G,\partial A)$, so
$i\bigl(\Cay(G,\partial A)\bigr)\leq 2$, and since $i\bigl(\Cay(G,\partial A)\bigr)\ge 2$, 
we get $i\bigl(\Cay(G,\partial A)\bigr)=2$.

Next, note that $\{e,ab,abt,b^2t\}$ is an independent set of size four.
We show that every independent set $I$ in $\Cay(G,\partial A)$ has at most four vertices.  
Since $\Cay(G,\partial A)$ is vertex-transitive, we may assume that $e\in I$.
The vertex set of $\Cay(G,\partial A)$ decomposes into the cosets
\[
T(h) = \{h, ah, a^2h\},\;h\in \langle b, t \rangle.
\]
Each $T(h)$ forms a triangle in $\Cay(G, \partial A)$ because $a,a^2\in \partial A$ and $a^3=e$. 
In an independent set, at most one vertex can be chosen from each triangle.
The triangles $T(t)$ and $T(bt)$ 
each contain exactly one vertex that is not adjacent to $e$: $at$ and $abt$, respectively.
If $I$ contains no vertices from these two triangles, then $|I| \leq 4$. 
Since $at$ and $abt$ are adjacent, at most one of them can belong to $I$. 
If $at \in I$, then we have 
\[
T(bt)\,\cup\,T(b) \subset N(e)\,\cup\, N(at);
\] 
if $abt \in I$, then 
\[
T(t)\,\cup\, T(b^2) \subset N(e)\,\cup\, N(abt). 
\]
In both cases, we see that at least two triangles cannot contain any elements of $I$.
Thus, $\alpha\bigl(\Cay(G,\partial A)\bigr)=4$.
\end{proof}

\section{Groups of order $16$}
\label{section:Order 16}

In this section, we compute the independence number and 
the independent domination number for Cayley graphs of certain groups of order $16$.
It is well known that there are $14$ groups of order $16$.
A modern and concise account of the classification of groups of order $16$ 
can be found in \cite{Wild}.
\begin{lemma}
\label{lemma:groups of order 16}
    All groups of order $16$ except $\mathbb{Z}_2^4$ are unstable.
\end{lemma}

\begin{proof}
Let $G$ be a group of order $16$. 
If $G$ contains an element of order $8$, then the statement follows from 
Lemma~\ref{lemma:indices 01 in Zn} 
and Corollary~\ref{corollary: index stable of subgroups}.
There are four non-abelian groups of order $16$ with this property:
\verb"C8 : C2", \verb"D16", \verb"QD16", and \verb"Q16".
Their respective IDs among groups of order $16$ in \GAP \cite{GAP} are 6, 7, 8, and 9, respectively.
Henceforth, all structural descriptions of groups (e.g., \verb"C8 : C2", etc.) 
and their corresponding IDs are directly taken from the \GAP computational algebra system.
Note that \GAP\unskip's \verb"StructureDescription" is not always unique up to isomorphism; in particular, 
\verb"(C4 x C2) : C2" occurs for IDs 3 and 13 among groups of order $16$. 
Hence, we will refer to such groups by their \GAP IDs when needed.

\smallskip\noindent
If $G$ has a normal subgroup $F$ of order $4$ with a cyclic quotient group $G/F$, then 
by Lemma~\ref{lemma:On cyclic quotient group}, $G$ is unstable.
There are three such non-abelian groups of order $16$: 
\verb"(C4 x C2) : C2", \verb"C4 : C4", and \verb"C8 : C2" 
(IDs 3, 4, and 6 in \GAP\unskip).
Every abelian group of order $16$ except $\mathbb{Z}_2^4$ also has 
a cyclic quotient group of order $4$, 
and hence it is unstable.

\smallskip\noindent
Thus, it remains to verify the claim for the three remaining groups of order $16$, 
namely those of \GAP IDs 11, 12, and 13, 
with \verb"StructureDescription" equal to \verb"C2 x D8", \verb"C2 x Q8", 
and \verb"(C4 x C2) : C2", respectively.
Now, for each of these three groups,
we specify a subset $A$ and compute the independence number and 
the independent domination number of the Cayley graph $\Cay(G,\partial A)$.

\begin{figure}[ht]
\centering
\InducedSubgraphsOfThreeGroups
\caption[Induced subgraphs on \protect{$G\setminus N[e]$} for three groups of order 16]
{Induced subgraphs on $G\setminus N[e]$ \\for Cayley graphs of three groups of order 16}
\label{fig:Induced subgraphs of groups of order 16}
\end{figure}

\smallskip\noindent
\textbf{Case} \verb"C2 x D8".
This group can be defined by the following presentation:
\[
G=\langle a,b,c\mid a^4=b^2=c^2=e, \;ac=ca,\; bc=cb,\; (ab)^2=e\rangle.
\]
Note that $D_4\cong\langle a,b\rangle$. If $A=\{e,a,b,c\}$, then
\[
\partial A=\{a,\; a^3,\; a^3b,\; b,\; ac,\; a^3c,\; bc,\; c\}.
\]
We have used that $(ab)^2=e$, equivalently $bab=a^3$, and that $bc=cb$.
In $\Cay(G,\partial A)$, consider the subgraph $X$ induced on the set
\[
G\setminus N[e]=\{a^2,\; a^2b,\; ab,\; a^2c,\; abc,\; a^2bc,\; a^3bc\}.
\]
Figure~\ref{fig:Induced subgraphs of groups of order 16}(i) shows $X$.
Since $i(X)=1$ and $\alpha(X)=3$, it follows from Lemma~\ref{lemma:vertex-transitive}
that 
$i\bigl(\Cay(G,\partial A)\bigr)=2$ and 
$\alpha\bigl(\Cay(G,\partial A)\bigr)=4$.

\smallskip\noindent
\textbf{Case} \verb"C2 x Q8".
Let 
\[
G=\langle a,b,c\mid a^4=b^4=c^2=e,\;a^2=b^2,\;ac=ca,\;bc=cb,\;(ab)^2=a^2\rangle.
\]
Let $H=\langle a,b\rangle$ and $A=\{e,a,b,c\}$. We have $H\cong Q_8$.
Since $a^2,c\in Z(G)$ and $ab^3=ba$, it follows that
\[
(A^{-1}A)\cap H=\{e,\; a,\; b,\; a^3,\; b^3,\; a^3b,\; b^3a\}\quad\text{and}\quad
H\setminus A^{-1}A=\{a^2\}
\]
and
\[
(A^{-1}A)\cap Hc=\{c,\; ac,\; bc,\; a^3c,\; b^3c\}\quad\text{and}\quad
Hc\setminus A^{-1}A=\{a^2c,\; abc,\; a^3bc\}.
\]
Therefore,
\[
\partial A=\{a,\; b,\; a^3,\; b^3,\; a^3b,\; b^3a,\; c,\; ac,\; bc,\; a^3c,\; b^3c\}
\]
and
\[
G\setminus N[e]=\{a^2,\; a^2c,\; abc,\; a^3bc\}.
\]
Let $X$ be the subgraph induced on $G\setminus N[e]$;
see Figure~\ref{fig:Induced subgraphs of groups of order 16}(ii).
It is clear that $i(X)=1$ and $\alpha(X)=3$.
Together with Lemma~\ref{lemma:vertex-transitive}, this implies
$i\bigl(\Cay(G,\partial A)\bigr)=2$ and $\alpha\bigl(\Cay(G,\partial A)\bigr)=4$.

\smallskip\noindent
\textbf{Case} \verb"(C4 x C2) : C2".
In this case, we consider the group of order $16$ with \GAP ID~13\unskip.
It admits the presentation:
\[
G=\langle a,b,c\mid a^4=b^2=c^2=e,\;ab=ba,\;ca=ac,\;(bc)^2=a^2\rangle.
\]

Note that, although this will not be used, $\langle b,c\rangle\cong D_4$ and 
$Q_8\cong\langle ab,ac\rangle$; 
moreover, $G$ is sometimes called the Pauli group (in the context of quantum physics).

Let $A=\{e,a,b,c\}$.
As above, set $H=\langle a,b\rangle$. We have $H\cong C_4\times C_2$. 
Since $a\in Z(G)$ and $cbc=a^2b$, it follows that
\[
(A^{-1}A)\cap H=\{e,\; a,\; a^3,\; b,\; ab,\; a^3b\}\quad\text{and}\quad
H\setminus A^{-1}A=\{a^2,\; a^2b\}
\]
and
\[
(A^{-1}A)\cap Hc=\{c,\; ac,\; a^3c,\; bc,\; a^2bc\}\quad\text{and}\quad
Hc\setminus A^{-1}A=\{a^2c,\; abc,\; a^3bc\}.
\]
Therefore,
\[
\partial A=\{a,\; a^3,\; ab,\; a^3b,\; b,\; c,\; ac,\; a^3c,\; bc,\; a^2bc\}
\]
and
\[
G\setminus N[e]=\{a^2,\; a^2b,\; a^2c,\; abc,\; a^3bc\}.
\]
The induced subgraph on $G\setminus N[e]$ is shown
in Figure~\ref{fig:Induced subgraphs of groups of order 16}(iii). 
Lemma~\ref{lemma:vertex-transitive} now yields 
$i\bigl(\Cay(G,\partial A)\bigr)=2$ and $\alpha\bigl(\Cay(G,\partial A)\bigr)=4$.
\end{proof}

\section{Characterizing finite stable groups}
\label{section:Characterise finite stable groups}
  
The following theorem gives a complete characterization of finite stable groups.
The proof of this theorem, as given in \cite[Theorem 3.1]{HooshmandArani}, 
relies essentially on computer computations.
Here, we present a proof that does not require any computer computations.

\begin{theorem}
\label{theorem:characterise finite stable groups} 
There are exactly $14$ finite stable groups up to isomorphism:
\begin{equation}\label{list:finite stable groups}
\begin{array}{cc}
C_1,\;C_2,\
C_2\times C_2,\
C_2\times C_2\times C_2,\
C_2\times C_2\times C_2\times C_2,\;\\
C_3,\
C_3\times C_3,\
C_4,\
C_2\times C_4,
C_5,\
C_7,\
S_3,\;D_4,\;Q_8.
\end{array}
\end{equation}
\end{theorem}

Using the notions of well-covered graphs and Cayley graphs, we can reformulate this theorem as follows:
\begin{theorem*}
    If $G$ is a finite group for which the Cayley graph $\Cay(G,\partial A)$ 
    is well-covered for every subset $A \subseteq G$, 
    then $G$ is one of the groups listed in \eqref{list:finite stable groups}.
    Conversely, if $G$ is one of these groups, then $\Cay(G,\partial A)$ is well-covered for every $A \subseteq G$.
\end{theorem*}

\begin{proof}
Let $G$ be a finite group satisfying the condition of the theorem.
By Lemma~\ref{lemma:indices 01 in Zn} and 
Corollary~\ref{corollary: index stable of subgroups}, 
every nontrivial cyclic subgroup of $G$ may only have order $2,\;3,\;4,\;5$, or $7$, 
and no other orders are allowed. 
Hence, if $p$ is a prime divisor of $|G|$, then $p\in\{2,3,5,7\}$. 
Furthermore, by Lemmas~\ref{lemma:Z2^5},~\ref{lemma:Z3^3}, and  
\ref{lemma:On cyclic quotient group},
the order of $G$ cannot be divisible by $49$, $25$, $27$, or $32$.

Moreover, the centralizer of every nontrivial element of $G$ must be a $p$-subgroup 
(for some $p\in\{2,3,5,7\}$); otherwise, $G$ would contain a cyclic subgroup of a forbidden order. 
In particular, if the center of $G$ is nontrivial, then $G$ is a $p$-group.

\smallskip\noindent
\textbf{Case 1:} $G$ is a $p$-group.
If $p\in\{5,7\}$, then $|G|=p$, and hence $G$ is cyclic. 
If $p=3$, then $G\cong C_3$ or $G\cong C_3\times C_3$.
If $p=2$, then by Lemma~\ref{lemma:groups of order 16} 
either $|G|\leq8$ or $G\cong C_2^4$.

\smallskip\noindent
\textbf{Case 2:} 
$|G|$ is divisible by $p\in\{5,7\}$ but is strictly larger than $p$.
Assume first $p=7$.
Let $P$ be a Sylow $p$-subgroup of $G$ and let $H=N_G(P)$.

Suppose that $H\neq P$.
Define $\varphi\colon H\to\Aut P$ by $\varphi(h)(x) = h^{-1} x h$ for $x\in P$.
Then $\Ker \varphi = C_H(P)$. 
Since $P\leq C_H(P)$ and $C_H(P)$ is a $p$-subgroup, since $P$ is a Sylow $p$-subgroup of $H$, 
we have $C_H(P)=P$. Hence, $H/P$ embeds into $\Aut(P)$.
It follows that $|H/P| \in \{2,3,6\}$, 
thus in any case $H$ has a subgroup of index $2$ or $3$ over $P$, 
and hence (as $C_H(P)=P$) contains a subgroup isomorphic to $D_7$ or a nonabelian group of order $21$.
None of these groups are stable 
(see Lemmas~\ref{lemma:indices Dn}, 
\ref{lemma:indices of group of order 21}), 
a contradiction.

Now assume that $H=P$.
Burnside's normal $p$-complement theorem 
\cite[Theorem 10.1.8]{Robinson} 
guarantees a normal subgroup $F\lhd G$ with $|G:F|=p$.  
By Lemma~\ref{lemma:On cyclic quotient group}, 
such a cyclic quotient prevents $G$ from being stable, 
another contradiction.
Thus, if $7\mid |G|$, then $|G|=7$.

An analogous argument applies when $p=5$: replace $D_7$ by $D_5$, and note that
any nontrivial factor of order $4$ would also yield an unstable group 
(by Lemmas~\ref{lemma:indices Dn} and~\ref{lemma:On cyclic quotient group}).

\smallskip\noindent
\textbf{Case 3:} 
$|G|$ is neither divisible by $7$ nor by $5$, but is divisible by $3$. 
We show that $G\cong S_3$ or $G\cong C_3$.  
Let $P$ be a Sylow 3-subgroup, and assume $G\neq P$.  
Then $P$ is elementary abelian of order $3$ or $9$
(since no element has order $9$).

\smallskip\noindent
\textbf{Case 3a:} $|P|=3$.
Let $Q$ be a Sylow $2$-subgroup of $G$. Then $|G:Q|=3$, 
so the action of $G$ on the left cosets of $Q$ yields a homomorphism $\varphi\colon G\to S_3$ 
whose kernel satisfies $\Ker(\varphi)\leq Q$.

If $\Ker(\varphi)=\{e\}$, then $G\cong\varphi(G)\leq S_3$, hence $G\cong S_3$ or $C_3$.
If $\Ker(\varphi)$ is nontrivial, 
let $A\lhd G$ be a minimal normal subgroup.
Then $A$ is an elementary abelian $2$-group,
and conjugation by a non-identity element of $P$ 
induces a fixed-point-free automorphism of order $3$ on $A$.
It follows that $|A|=4$, which leads to a contradiction via 
Lemma~\ref{lemma:semidirect Z2^2:Z3}.

\smallskip\noindent
\textbf{Case 3b:} $|P|=9$.
If $N_G(P)\neq P$, then $N_G(P)/P$ has even order; 
in particular, there exists $a\in N_G(P)$ of order $2$.
Set $H=P\langle a\rangle$.
Then $C_H(a)=\langle a\rangle$(since $C_G(a)$ is a $2$-group, so $C_P(a)=\{e\}$),
and $H$ is the group in 
Lemma~\ref{lemma:semidirect Z3^2:Z2}, hence it is unstable.

If $N_G(P)=P$, let $Q$ be a Sylow $2$-subgroup, so that $G=QP$.
Since $|P|=9$ and $P$ is elementary abelian, every non-identity element of $P$ has order $3$.
For any element $t\in G$ of order $3$, the centralizer $C_G(t)$ is a $3$-subgroup.
In particular, distinct Sylow $3$-subgroups intersect trivially, because
if $t\ne e$ lies in the intersection, then both Sylow $3$-subgroups are contained in $C_G(t)$ 
(since they are abelian), and hence they coincide.
Since $N_G(P)=P$, the number of Sylow $3$-subgroups equals $[G:N_G(P)]=[G:P]=|Q|$.
Therefore $G$ has exactly $8|Q|$ elements of order $3$, 
and hence exactly $|G|-8|Q|=|Q|$ elements of $2$-power order.
It follows that the Sylow $2$-subgroup is unique, and thus $Q\lhd G$.
(Alternatively, $Q\lhd G$ follows immediately from Burnside's normal $3$-complement theorem, 
since here $N_G(P)=C_G(P)=P$.)

Now fix any non-identity $t\in P$. Since $Q$ is a $2$-group, 
the Sylow $3$-subgroup of $Q\langle t\rangle$ is $\langle t\rangle$ and has order $3$.
Hence $Q\langle t\rangle$ is unstable by Case~3a.

\medskip
To complete the proof of the theorem, it remains to verify that each group in the list 
\eqref{list:finite stable groups} is stable. 
We begin with several preliminary facts and then apply them to the groups in question.
Let $A\subseteq G$ be nonempty.

\smallskip\noindent
\textbf{Step 1.} Basic observations.
If $A = \{e\}$, then $\partial A = \varnothing$ and $\Cay(G, \partial A)$ has no edges
and is therefore well-covered.
Let $H = \langle A \rangle$. If $\Cay(H,\partial A)$ is well-covered, 
then $\Cay(G,\partial A)$ is also well-covered 
(see Lemma~\ref{lemma:On the subgroup index}). 
Consequently, it suffices to consider generating subsets $A$ of $G$ with $e\in A$.

\smallskip\noindent
\textbf{Step 2.} Complete induced subgraphs.     
If $G=A^{-1}A$, then $\Cay(G,\partial A)$ is a complete graph 
and is trivially well-covered.
Assume $G\neq A^{-1}A$, and let $X$ be the subgraph of $\Cay(G,\partial A)$ 
induced on $G\setminus A^{-1}A$. Recall, that $A^{-1}A=N[e]$ and $\partial A=N(e)$.
If $X$ is well-covered, then so is $\Cay(G,\partial A)$ by Lemma~\ref{lemma:vertex-transitive}. 
In particular, if $|X|\leq2$, the graph $X$ is trivially well-covered, 
and hence $\Cay(G,\partial A)$ is well-covered. 
Moreover, if $X$ is a complete graph, then for any $B \supseteq A$, 
the subgraph induced by $G\setminus B^{-1}B$ is again complete (or empty), 
and hence $\Cay(G,\partial B)$ is well-covered.

\smallskip\noindent
\textbf{Step 3.} Groups of very small order. 
If $|G|\leq5$ and $A$ is a generating subset of $G$, it is straightforward to check 
that $|G\setminus A^{-1}A|\leq2$. Hence $\Cay(G,\partial A)$ is well-covered. 
This shows that all cyclic groups $C_i$ (with $i\leq5$) and 
the group $C_2 \times C_2$ are stable.  
   
\smallskip\noindent
\textbf{Step 4.} The groups
$S_3$, $C_3\times C_3$, $C_2\times C_4$, $C_2^3$, $D_4$, and $Q_8$.
Let $A\subseteq G$ be a generating subset with $e\in A$, and assume that $A$ is minimal with this property.
Since $G$ is not cyclic, we have $|A|\geq3$.

\smallskip\noindent
\textbf{Step 4a.} $G=Q_8$.
Any minimal generating set of the group $Q_8$ consists of two elements of order $>2$.
If $A=\{e,x,y\}$ with $|x|,|y|>2$, then 
\begin{equation}\label{list:dA33}
  A^{-1}A=\{e,x,y,x^{-1},y^{-1},x^{-1}y,y^{-1}x\}.
\end{equation}
In particular, $|A^{-1}A|\geq6$, and hence $|G\setminus A^{-1}A|\leq2$. 
By Step~2, $\Cay(G,\partial A)$ is well-covered.
Thus $G$ is stable.

\smallskip\noindent
\textbf{Step 4b.} $G=C_3\times C_3$.
From \eqref{list:dA33} it follows in this case that $|A^{-1}A|=7$; 
hence $\Cay(G,\partial A)$ is well-covered and $G$ is stable.

\smallskip\noindent
\textbf{Step 4c.} $G=C_2\times C_4$.
If $A=\{e,x,y\}$ with $|x|=2$ and $|y|>2$, then 
\begin{equation}\label{list:dA23}
  A^{-1}A=\{e,x,y,y^{-1},xy,y^{-1}x\}.
\end{equation}
Since $G$ is abelian and $|y|=4$, the elements of the list \eqref{list:dA23} are distinct, 
and we have $|A^{-1}A|=6$.
Thus whether a minimal generating set consists solely of elements of order $4$ (see \eqref{list:dA33}) 
or includes a generator of order $2$ (as \eqref{list:dA23}), $\Cay(G,\partial A)$ is well-covered. 
Hence $G$ is stable.

\smallskip\noindent
\textbf{Step 4d.} $G=S_3$.
In this case, a minimal generating set contains an element of order $2$ and an element of order $3$, 
or both generators have order $2$.
If $A=\{e,x,y\}$ with $|x|=|y|=2$, then 
\begin{equation}\label{list:dA22}
  A^{-1}A=\{e,x,y,xy,yx\}.
\end{equation}
A direct check shows that $|S_3\setminus A^{-1}A|\leq2$ for every generating subset $A$ with $e\in A$. 
Hence $S_3$ is stable by Step~2.

\smallskip\noindent
\textbf{Step 4e.} $G=C_2^3$.
Let $A\subseteq G$ be a minimal generating subset with $e\in A$.
Then $A=\{e,x,y,z\}$, where $x,y,z$ are independent.
Since $G\cong C_2^3$ is elementary abelian, we have $A^{-1}A=AA=\{e,x,y,z,xy,xz,yz\}$,
and hence $|A^{-1}A|=7$ and $|G\setminus A^{-1}A|=1$ for every such $A$.
Thus $G$ is stable by Step~2.
 
\smallskip\noindent
\textbf{Step 4f.} $G=D_4$.
Let $A=\{e,x,y\}$ be a minimal generating subset with $e\in A$.
Then either $|x|=|y|=2$ or, up to swapping $x$ and $y$, we have $|x|=2$ and $|y|=4$.
We treat the case $|x|=|y|=2$; 
the second case yields the same Cayley graph after relabeling the generators.
By \eqref{list:dA22}, we have
\[
X=G\setminus A^{-1}A=\{xyx,\;yxy,\;(xy)^2\}.
\]
Moreover, each product of two distinct vertices of $X$ lies in $\partial A$:
\[
xyx\,(yxy)^{-1}=yx\in\partial A,\qquad
xyx\,(xy)^{-2}=y\in\partial A,\qquad
yxy\,(xy)^{-2}=x\in\partial A.
\]
Hence the induced subgraph on $X$ is $K_3$, and therefore $\Cay(G,\partial A)$ is well-covered by Step~2.
Thus $D_4$ is stable.

\smallskip\noindent
\textbf{Step 5.} Cyclic group of order $7$.  
Let $G=\langle a\rangle$ be cyclic of order $7$.
If $A=\{e, a\}$, then by Lemma~\ref{lemma:indices 01 in Zn}, $\Cay(G,\partial A)$ is well-covered.      
If $|A|\geq3$, then $|G\setminus A^{-1}A|\leq2$, 
and hence $\Cay(G,\partial A)$ is well-covered by Step~2. 
Thus $C_7$ is stable.

\smallskip\noindent
\textbf{Step 6.} The group $G=C_2^4\cong \mathbb{F}_2^4$. 
Write $ijkl$ for $(i,j,k,l)\in\mathbb{F}_2^4$. Consider  
\[
  A = \{0000,\,1000,\,0100,\,0010,\,0001\}.
\]
Then $\partial A$ consists of these five elements together with 
all elements having exactly two ones.
Consequently, $X=G\setminus (A + A)$ is the set of vectors with exactly three ones, 
together with $1111$; thus $|X|=5$.
The subgraph induced on $X$ is complete and hence well-covered,
so $\Cay(G,\partial A)$ is well-covered.
   
Now let $B\subseteq G$ be any generating subset.  
Since $\mathbb{F}_2^4$ is a $4$-dimensional vector space over $\mathbb{F}_2$, 
$B$ must contain at least four elements forming a basis.  
Without loss of generality, assume $B$ includes $0000$, $1000$, $0100$, $0010$, and $0001$.  
We have already seen that the subgraph induced by $G \setminus (A + A)$ is complete; 
hence $G\setminus (B + B)$ also induces a (possibly empty) complete subgraph.  
It follows that $\Cay(G, \partial B)$ is well-covered, and thus $G$ is stable.
This completes the proof.
\end{proof}

\section{\GAP verification for Cayley graphs $\Cay(G,S)$}
\label{section:gap}

This section provides a short \GAP script to compute 
the independent domination number $i(\Gamma)$ and 
the independence number $\alpha(\Gamma)$ for the Cayley graphs considered in the paper.
It is intended as a computational check for readers who 
do not wish to follow the arguments in the main text.

For each group $G$ we fix a subset $A\subseteq G$ (specified case-by-case in the code) and define
\[
S=\{x^{-1}y:\ x,y\in A,\ x\neq y\}.
\]
We then form the Cayley digraph $\Gamma=\Cay(G,S)$ using the \GAP package \texttt{digraphs}
in \cite{Digraphs}.
In our examples $S$ is inverse-closed, 
hence $\Gamma$ may be viewed as an undirected Cayley graph.

The script enumerates all maximal independent sets of $\Gamma$ and 
records the minimum and maximum of their sizes.
Since $\Gamma$ is undirected, maximal independent sets coincide with independent dominating sets; 
hence the minimum (resp. maximum) size among maximal independent sets 
equals $i(\Gamma)$ (resp. $\alpha(\Gamma)$).
Below is a step-by-step guide on how to use the \GAP code provided here 
to compute $i(\Gamma)$ and 
$\alpha(\Gamma)$, assuming that \GAP is already installed and running on your computer.

\smallskip\noindent
\textbf{Step 1.} Load the \texttt{digraphs} package by entering the command
\begin{lstlisting}[gobble=0]
   LoadPackage("digraphs");;
\end{lstlisting}

\smallskip\noindent
\textbf{Step 2.} Enter the only user-defined function. 
For a given group $G$ and a given subset $A\subseteq G$, 
the function constructs the Cayley graph $\Gamma=\Cay(G,\partial A)$,  
and computes $i(\Gamma)$ and $\alpha(\Gamma)$.
This function computes the set $\partial A$ and the set $D = G\setminus A^{-1}A$.   
The set $D$ is not used in the computations, 
but both $\partial A$ and $D$ can be inspected simply by entering the \GAP commands  \verb"r.S;" and \verb"r.D;". 
Here $S=\partial A$ in the code.

\begin{lstlisting}[gobble=0]
PrintIndepInvariants := function(G, A)
    local S, D, Cay, sizes;
    S := Set(List(Cartesian(A, A), x -> x[1]^-1 * x[2]));
    D := Difference(Elements(G), S);
    S := Difference(S, [One(G)]);
    Cay := CayleyDigraph(G, S);
    sizes := List(DigraphMaximalIndependentSets(Cay), Size);
    Print("StructureDescription(G) = \"", StructureDescription(G), "\"; ",
          "i = ", Minimum(sizes), ", alpha = ", Maximum(sizes), ".", "\n\n");
    return rec(S := S, D := D);
end;
\end{lstlisting}

\smallskip\noindent
\textbf{Step 3.}
For each of the eight groups listed below, enter the corresponding \GAP commands.

\begin{enumerate}
  \item The nonabelian group $C_7 \rtimes C_3$ of order $21$.  
    \begin{lstlisting}        
        F := FreeGroup("a", "b");; a := F.1;; b := F.2;;
        G := F / [ b^3, a^7, b*a*b^-1*a^-2 ];;
        a := G.1;; b := G.2;;
        A := [One(G), a, b];;
        r := PrintIndepInvariants(G, A);;
    \end{lstlisting}
    Output:
    \lstinline|StructureDescription(G) = "C7 : C3"; i = 3, alpha = 6.|

  \item The unitriangular matrix group  $UT(3,3)$ of degree three over the field with 
   $3$ elements.
   \begin{lstlisting}        
        F := FreeGroup("a", "b", "c");; a := F.1;; b := F.2;; c := F.3;;
        G := F / [ a^3, b^3, c^3, 
                   b^-1*a^-1*b*a*c^-1, a^-1*c^-1*a*c, b^-1*c^-1*b*c];;
        a := G.1;; b := G.2;; c := G.3;;
        A := [One(G), a, b, c];;
        r := PrintIndepInvariants(G, A);;
    \end{lstlisting}
    Output:
    \lstinline|StructureDescription(G) = "(C3 x C3) : C3"; i = 3, alpha = 6.|

   \item The elementary abelian group $C_3^3$ of order $27$.
    \begin{lstlisting}        
        F := FreeGroup("e1", "e2", "e3");; e1 := F.1;; e2 := F.2;; e3 := F.3;;
        G := F / [ e1^3, e2^3, e3^3, 
                   e1^-1*e2^-1*e1*e2, e1^-1*e3^-1*e1*e3, e2^-1*e3^-1*e2*e3 ];;
        e1 := G.1;; e2 := G.2;; e3 := G.3;;
        A := [One(G), e1, e2, e3];;
        r := PrintIndepInvariants(G, A);;
    \end{lstlisting}
    Output: 
    \lstinline|StructureDescription(G) = "C3 x C3 x C3"; i = 3, alpha = 4.|

  \item The alternating group $A_4$.
    \begin{lstlisting}        
        F := FreeGroup("a", "b", "t");; a := F.1;; b := F.2;; t := F.3;;
        G := F / [ a^2, b^2, t^3, 
                   a^-1*b^-1*a*b, t^-1*a*t*b^-1, t^-1*b*t*b^-1*a^-1 ];;
        a := G.1;; b := G.2;; t := G.3;;
        A := [One(G), b, t];;
        r := PrintIndepInvariants(G, A);;
    \end{lstlisting}      
    Output:
    \lstinline|StructureDescription(G) = "A4"; i = 2, alpha = 3.|
  
  \item  The semidirect product $(C_3\times C_3)\rtimes C_2$. 
  Here, the non-identity element of $C_2$ acts on $C_3\times C_3$ by inverting elements.
    \begin{lstlisting}
        F := FreeGroup("a", "b", "t");; a := F.1;; b := F.2;; t := F.3;;
        G := F / [ a^3, b^3, t^2, a^-1*b^-1*a*b, t*a*t*a, t*b*t*b];;
        a := G.1;; b := G.2;; t := G.3;;
        A := [One(G), a, t, b*t];;
        r := PrintIndepInvariants(G, A);;
    \end{lstlisting}      
    Output:
    \lstinline|StructureDescription(G) = "(C3 x C3) : C2"; i = 2, alpha = 4.|

  \item The direct product $C_2\times D_4$, the group of order $16$ with \GAP ID 11.
    \begin{lstlisting}
        F := FreeGroup("a", "b", "c");; a := F.1;; b := F.2;; c := F.3;;
        G := F / [ a^4, b^2, c^2, a^-1*c^-1*a*c,  b^-1*c^-1*b*c, (a*b)^2];;
        a := G.1;; b := G.2;; c := G.3;;
        A := [One(G), a, b, c];;
        r := PrintIndepInvariants(G, A);;
    \end{lstlisting} 
    Output:
    \lstinline|StructureDescription(G) = "C2 x D8"; i = 2, alpha = 4.|

  \item The direct product $C_2\times Q_8$, the group of order $16$ with \GAP ID 12.
    \begin{lstlisting}
        F := FreeGroup("a", "b", "c");; a := F.1;; b := F.2;; c := F.3;;
        G := F / [ a^4, b^4, c^2, a^2*b^2, 
                                    a^-1*c^-1*a*c, b^-1*c^-1*b*c, (a*b)^2*a^2];;
        a := G.1;; b := G.2;; c := G.3;;
        A := [One(G), a, b, c];;
        r := PrintIndepInvariants(G, A);;
    \end{lstlisting}  
    Output:
    \lstinline|StructureDescription(G) = "C2 x Q8"; i = 2, alpha = 4.|

  \item The group of order 16: \GAP ID 13.
    \begin{lstlisting}
        F := FreeGroup("a", "b", "c");; a := F.1;; b := F.2;; c := F.3;;
        G := F / [ a^4, b^2, c^2, a^-1*b^-1*a*b, a^-1*c^-1*a*c, (b*c)^2*a^2];;
        a := G.1;; b := G.2;; c := G.3;;
        A := [One(G), a, b, c];;
        r := PrintIndepInvariants(G, A);;
    \end{lstlisting}  
    Output:
    \lstinline|StructureDescription(G) = "(C4 x C2) : C2"; i = 2, alpha = 4.|
\end{enumerate}

\end{document}